\documentclass[reqno]{amsart}

\newcommand{\Rm}{\mathbb{R}}
\newcommand{\Cm}{\mathbb{C}}

\newcommand{\Sm}{\mathbb{S}}
\newcommand{\be}{\begin{equation}}
\newcommand{\ee}{\end{equation}}
\newcommand{\ba}{\begin{equation}\begin{aligned}}
\newcommand{\ea}{\end{aligned}\end{equation}}
\newcommand{\va}{\varphi}
\newcommand{\vth}{\vartheta}
\newcommand{\pp}{\partial}
\newcommand{\vv}[1]{\boldsymbol{\mathrm{#1}}}
\newcommand{\hvv}[1]{\boldsymbol{\hat{\mathrm{#1}}}}
\newcommand{\uv}{\boldsymbol{{\hat{\mathrm{s}}}}}
\newcommand{\uvk}{\boldsymbol{\hat{\mathrm{k}}}}

\newcommand{\rrf}[1]{\mathop{\mathcal{R}_{{#1}}}}
\newcommand{\irrf}[1]{\mathop{\mathcal{R}_{{#1}}^{-1}}}

\usepackage{amsmath}
\usepackage[foot]{amsaddr}
\usepackage{amsthm,amssymb,mathrsfs}
\usepackage{graphicx}

\newtheorem{thm}{Theorem}[section]
\newtheorem{lem}[thm]{Lemma}

\theoremstyle{remark}

\title[]{Three-dimensional analytical discrete-ordinates method for the radiative transport equation}

\author{Manabu Machida$^{1,2}$}
\thanks{CONTACT M. Machida. Email: machida@hama-med.ac.jp}

\author{Kaustav Das$^{1,\dagger}$}
\thanks{$\dagger$Current Address: Department of Biomedical Engineering, Tokyo University of Agriculture and Technology, Koganei, Tokyo 184-8588, Japan (kdas@go.tuat.ac.jp)}

\address{$^1$
Institute for Medical Photonics Research, 
Hamamatsu University School of Medicine, 
Hamamatsu 431-3192, Japan}

\address{$^2$
JST, PRESTO, Kawaguchi, Saitama 332-0012, Japan}

\begin{document}

\begin{abstract}
The radiative transport equation in a three-dimensional infinite medium is considered. The coefficients of the radiative transport equation are assumed to be constant. For a pencil beam, we extend the analytical discrete-ordinates method to three dimensions by rotating reference frames. Obtained results are compared to Monte Carlo simulation.
\end{abstract}

\maketitle

\section{Introduction}
\label{intro}

The discrete-ordinates method was first proposed by Wick \cite{Wick43} and Chandrasekhar \cite{Chandrasekhar44}. In one-dimensional transport theory, analytical solutions are available in terms of singular eigenfunctions by Case's method \cite{Case60,Case-Zweifel,McCormick-Kuscer66,Mika61}. The method of analytical discrete ordinates (ADO) \cite{Barichello11,Barichello-Siewert99a,Barichello-Garcia-Siewert00,Barichello-Siewert02} assumes the same separated solution as Case's method for the equation in which the cosine of the polar angle is discretized. Although the method was constructed for the one-dimensional transport equation, the searchlight problem in three dimensions was considered with ADO in the case of isotropic scattering \cite{Barichello-Siewert00}. ADO was developed for the one-dimensional transport equation with the scattering phase function which has the polar and azimuthal angles \cite{Garcia-Siewert10}.

In this paper we consider the three-dimensional radiative transport equation with general anisotropic scattering and construct ADO in three dimensions with the help of rotated reference frames \cite{Markel04}. We assume that the medium is homogeneous, and the scattering and absorption coefficients of the radiative transport equation are constants. Furthermore we assume the pencil beam for the source term. The formulation is a direct extension of ADO developed in one dimension and the introduction of a pseudoproblem is not necessary. This is made possible by the knowledge of three-dimensional singular eigenfunctions obtained from the extension of Case's method to three dimensions \cite{Machida14}. For isotropic scattering in an infinite medium, solutions are known for different sources \cite{Ganapol-Kornreich95}. General anisotropic scattering can be treated in the numerical scheme proposed below.

Kim and his collaborators developed discrete ordinates for the three-dimensional radiative transport equation \cite{Kim04,Kim-Keller03,Kim-Moscoso04}. The present method is similar to their formulation in the sense that the same structure of the plane-wave decomposition is derived and discrete ordinates are used. However, in the three-dimensional ADO which is developed in this paper, the explicit expression of eigenmodes are obtained and there is no need of computing eigenvalues for each Fourier vector.

The rotated reference frames were first introduced in the spherical-harmonics method \cite{Dede64,Kobayashi77,Markel04}. In particular, Markel devised an efficient numerical algorithm called the method of rotated reference frames \cite{Markel04,Panasyuk06,Machida-etal10}. In the three-dimensional ADO, eigenmodes are obtained by the rotation of the eigenmodes for the original ADO in one-dimensional transport theory.

The remainder of the paper is organized as follows. In Sec.~\ref{rte}, the radiative transport equation is introduced. The specific intensity is then decomposed into the ballistic part and scattering part. In Sec.~\ref{modes}, eigenmodes are considered. Using these eigenmodes, the scattering part of the specific intensity is calculated in Sec.~\ref{Ispart}. The formula for the energy density is derived in Sec.~\ref{eden}. In Sec.~\ref{num}, numerical tests are performed. Concluding remarks are given in Sec.~\ref{concl}. The fundamental solution is calculated in Appendix \ref{green}.

\section{Radiative transport equation}
\label{rte}

We consider the radiative transport equation in the three-dimensional space $\Rm^3$. Let $\vv{r}\in\Rm^3$ be the position vector with $\vv{r}={^t}(\vv{\rho},z)$, $\vv{\rho}={^t}(x,y)$. The direction of propagation is specified by unit vector $\uv\in\Sm^2$. The specific intensity $I(\vv{r},\uv)$ obeys the following radiative transport equation.
\begin{equation}
\left(\uv\cdot\nabla+\mu_t\right)I(\vv{r},\uv)=
\mu_s\int_{\Sm^2}p(\uv,\uv')I(\vv{r},\uv')\,d\uv'+f(\vv{r},\uv),
\label{rte0}
\end{equation}
where $\nabla=\pp/\pp\vv{r}$, and $\mu_t>0$ and $\mu_s>0$ are constants. Moreover we assume that $\mu_a=\mu_t-\mu_s$ is positive. We have $I(\vv{r},\uv)\to0$ as $|\vv{r}|\to\infty$. Here, $p(\uv,\uv')$ is the scattering phase function which is assumed to be
\be
p(\uv,\uv')=
\frac{1}{4\pi}\sum_{l=0}^{l_{\rm max}}(2l+1){\rm g}^lP_l(\uv\cdot\uv')=
\sum_{l=0}^{l_{\rm max}}\sum_{m=-l}^l{\rm g}^lY_{lm}(\uv)Y_{lm}^*(\uv')
\ee
with $l_{\rm max}\ge0$, Legendre polynomials $P_l$, and spherical harmonics $Y_{lm}$. Spherical harmonics are defined as
\be
Y_{lm}(\uv)=Y_{lm}(\mu,\va)=\sqrt{\frac{2l+1}{4\pi}\frac{(l-m)!}{(l+m)!}}P_l^m(\mu)e^{im\va},
\ee
where $\mu\in[-1,1]$ is the cosine of the polar angle of $\uv$, $\va\in[0,2\pi)$ is the azimuthal angle of $\uv$, and $P_l^m(\mu)$ are associated Legendre polynomials. The constant ${\rm g}\in[0,1]$ is the anisotropic factor. We will consider the following source term.
\begin{equation}
f(\vv{r},\uv)=\delta(\vv{\rho})\delta(z)\delta(\uv-\uv_0),
\label{source1}
\end{equation}
where $\delta(\vv{\rho})=\delta(x)\delta(y)$ and $\uv_0\in\Sm^2$. That is, $I(\vv{r},\uv)$ is the fundamental solution of the radiative transport equation. Here, $\uv\in\Sm^2$ is given by
\be
\uv=\left(\begin{array}{c}\vv{\omega} \\ \mu\end{array}\right),
\qquad
\vv{\omega}=\left(\begin{array}{c}
\sqrt{1-\mu^2}\cos\va \\ \sqrt{1-\mu^2}\sin\va
\end{array}\right)
\ee
for $-1\le\mu\le1$, $0\le\va<2\pi$. We note that $\delta(\uv-\uv_0)=\delta(\mu-\mu_0)\delta(\va-\va_0)$ with $\mu_0,\va_0$ angles for $\uv_0$.

By dividing both sides of the above radiative transport equation (\ref{rte0}) by $\mu_t$, we obtain
\be
\left(\uv\cdot\nabla_*+1\right)I_*(\vv{r}_*,\uv)=
\varpi\int_{\Sm^2}p(\uv,\uv')I_*(\vv{r}_*,\uv')\,d\uv'+f_*(\vv{r}_*,\uv),
%\label{rte1}
\ee
where $\vv{r}_*=\mu_t\vv{r}$, $\nabla_*=\pp/\pp\vv{r}_*$, $I_*(\vv{r}_*,\uv)=I(\vv{r},\uv)$, and
\be
f_*(\vv{r}_*,\uv)=\frac{1}{\mu_t}f(\vv{r},\uv)=
\mu_t^2\delta(\vv{\rho}_*)\delta(z_*)\delta(\uv-\uv_0).
\ee
Here, $\varpi=\mu_s/\mu_t\in(0,1)$ is the albedo for single scattering. Hereafter we will drop the subscript $*$.

We decompose the specific intensity into the ballistic and scattering terms as
\begin{equation}
I(\vv{r},\uv)=I_b(\vv{r},\uv)+I_s(\vv{r},\uv).
\label{Idecomp}
\end{equation}
The ballistic term $I_b$ satisfies
\be
\left(\uv\cdot\nabla+1\right)I_b(\vv{r},\uv)=f(\vv{r},\uv)
\quad\mbox{in}\;\Rm^3\times\Sm^2,
%\label{subt:rte0}
\ee
The scattering term $I_s$ satisfies
\begin{equation}
\left(\uv\cdot\nabla+1\right)I_s(\vv{r},\uv)=
\varpi\int_{\Sm^2}p(\uv,\uv')I_s(\vv{r},\uv')\,d\uv'+
S(\vv{r},\uv)
\quad\mbox{in}\;\Rm^3\times\Sm^2,
\label{subt:rte}
\end{equation}
where
\be
S(\vv{r},\uv)=
\varpi\int_{\Sm^2}p(\uv,\uv')I_b(\vv{r},\uv')\,d\uv'.
\ee
Let us introduce the step function $\Theta$ as $\Theta(x)=1$ for $x>0$ and $\Theta(x)=0$ for $x<0$. We obtain
\be
I_b(\vv{r},\uv)=
\Theta(\mu_0z)\frac{\mu_t^2}{|\mu_0|}\delta\left(\vv{\rho}-\vv{\omega}_0z/\mu_0\right)e^{-z/\mu_0}\delta(\uv-\uv_0),
\ee
where $\vv{\omega}_0,\mu_0$ are components of $\uv_0$. We note that
\ba
\delta\left(\vv{\rho}-\vv{\omega}z/\mu\right)
&=
\frac{1}{|\vv{\rho}|}\delta\left(|\vv{\rho}|-z\frac{\sqrt{1-\mu^2}}{\mu}\right)
\delta\left(\va_{\vv{\rho}}-\va\right)
\\
&=
\left|\frac{z}{(\rho^2+z^2)^{\frac{3}{2}}}\right|\left[\delta\left(\mu-\frac{z}{\sqrt{\rho^2+z^2}}\right)+\delta\left(\mu+\frac{z}{\sqrt{\rho^2+z^2}}\right)\right]
\delta\left(\va_{\vv{\rho}}-\va\right),
\ea
where $\va_{\vv{\rho}}$ is the angle of $\vv{\rho}$ and $\rho=|\vv{\rho}|$. Hence,
\be
S(\vv{r},\uv)=
\varpi p(\uv,\uv_0)\Theta(\mu_0z)\frac{\mu_t^2}{|\mu_0|}\delta\left(\vv{\rho}-\vv{\omega}_0z/\mu_0\right)e^{-z/\mu_0}.
\ee
The Fourier transform $\widetilde{I}_s(\vv{q},z,\uv)$ is given by
\be
\widetilde{I}_s(\vv{q},z,\uv)=
\int_{\Rm^2}e^{-i\vv{q}\cdot\vv{\rho}}I_s(\vv{r},\uv)\,d\vv{\rho},
\ee
where
\be
\vv{q}=q\begin{pmatrix}\cos\va_{\vv{q}}\\\sin\va_{\vv{q}}\end{pmatrix},\quad
q=|\vv{q}|,\quad 0\le\va_{\vv{q}}<2\pi.
\ee
This $\widetilde{I}_s(\vv{q},z,\uv)$ is obtained as the solution to the following equation.
\begin{equation}
\left(\mu\frac{\pp}{\pp z}+1+i\vv{\omega}\cdot\vv{q}\right)
\widetilde{I}_s=
\varpi\int_{\Sm^2}p(\uv,\uv')\widetilde{I}_s(\vv{q},z,\uv')\,d\uv'+
\widetilde{S},\quad z\in\Rm,\;\uv\in\Sm^2,
\label{intro:I2Feq}
\end{equation}
where
\begin{equation}
\widetilde{S}(\vv{q},z,\uv)=
\varpi p(\uv,\uv_0)\Theta(\mu_0z)\frac{\mu_t^2}{|\mu_0|}e^{-z/\mu_0}e^{-i\vv{q}\cdot\vv{\omega}_0z/\mu_0}.
\label{search:tildeS}
\end{equation}
The scattering term is given by
\begin{equation}
I_s(\vv{r},\uv)=I_s(\vv{\rho},z,\uv)=
\frac{1}{(2\pi)^2}\int_{\Rm^2}e^{i\vv{q}\cdot\vv{\rho}}
\widetilde{I}_s(\vv{q},z,\uv)\,d\vv{q}.
\label{intro:I2fourier}
\end{equation}
We note that
\be
\vv{\omega}\cdot\vv{q}=q\sqrt{1-\mu^2}\cos(\va-\va_{\vv{q}}).
\ee

We will obtain $\widetilde{I}_s$ by using discrete ordinates:
\be
\widetilde{I}_s\approx\check{I}_s,
\ee
where $\check{I}_s$ is the solution to the following equation.
\begin{equation}
\left(\mu_i\frac{\pp}{\pp z}+1+i\vv{\omega}_i\cdot\vv{q}\right)
\check{I}_s=
\varpi\sum_{i'=1}^{2N}w_{i'}\int_0^{2\pi}p(\uv_i,\uv_{i'})\check{I}_s(\vv{q},z,\uv_{i'})\,d\va'+
\widetilde{S}
\label{intro:I2do}
\end{equation}
for $z\in\Rm$, $\uv_i\in\Sm^2$. Here, we discretize the integral by discrete ordinates with the Gauss-Legendre quadrature of weights $w_i$ ($i=1,\dots,N$). We use the Golub-Welsch algorithm \cite{GW}. We label $\mu_i$ ($i=1,\dots,N$) such that $0<\mu_1<\mu_2<\cdots<\mu_N<1$ and $-1<\mu_{2N}<\cdots<\mu_{N+2}<\mu_{N+1}<0$ ($\mu_{N+i}=-\mu_i$, $1\le i\le N$). Furthermore we introduced
\be
\uv_i=
\begin{pmatrix}\vv{\omega}_i \\ \mu_i\end{pmatrix},
\quad
\vv{\omega}_i=\begin{pmatrix}
\sqrt{1-\mu_i^2}\cos\va \\ \sqrt{1-\mu_i^2}\sin\va
\end{pmatrix}
\ee
for $i=1,2,\dots,2N$, $0\le\va<2\pi$. We have
\be
\vv{\omega}_i\cdot\vv{q}=
q\sqrt{1-\mu_i^2}\cos\left(\va-\va_{\vv{q}}\right).
\ee
Let us define
\be
\omega_i^0(q,\va)=q\sqrt{1-\mu_i^2}\cos\va.
\ee

\section{Eigenmodes}
\label{modes}

In order to obtain $\check{I}_s$ in (\ref{intro:I2do}), we find the solution $\check{I}$ to (\ref{homog}) below with separation of variables. We refer to $\check{I}$ as eigenmodes. The eigenmodes are labeled by a separation constant $\nu$. This constant $\nu$, which is referred to as an eigenvalue, is the reciprocal of the square roof of an eigenvalue of the eigenproblem (\ref{evalsys}) below or the negative of the reciprocal.

We begin by defining unit vector $\uvk$ as \cite{Markel04}
\be
\uvk=\uvk(\nu,\vv{q})=
\left(\begin{array}{c}-i\nu\vv{q}\\\hat{k}_z(\nu q)\end{array}\right),
\quad
\hat{k}_z(\nu q)=\sqrt{1+(\nu q)^2},
\ee
where $\nu\in\Rm$. We will extend the rotation in $\Rm^3$ to $\Cm^3$ by analytic continuation (see (\ref{rotWig}) below) \cite{Panasyuk06,Machida-etal10}. To this end, we extend the usual dot product in $\Rm^3$ to $\Cm^3$ as $\vv{\alpha}\cdot\vv{\beta}=\alpha_1\beta_1+\alpha_2\beta_2+\alpha_3\beta_3$ for $\vv{\alpha},\vv{\beta}\in\Cm^3$. We note that $\uvk\cdot\uvk=1$. We have
\begin{equation}
\uv_i\cdot\uvk=
-i\nu q\sqrt{1-\mu_i^2}\cos(\va-\va_{\vv{q}})+\hat{k}_z(\nu q)\mu_i.
\label{modes:sk}
\end{equation}

Let us consider the homogeneous equation below (i.e., the equation (\ref{intro:I2do}) without the source term $\widetilde{S}$).
\begin{equation}
\left(\mu_i\frac{\pp}{\pp z}+1+i\vv{\omega}_i\cdot\vv{q}\right)
\check{I}(\vv{q},z,\uv_i)=
\varpi\sum_{i'=1}^{2N}w_{i'}\int_0^{2\pi}p(\uv_i,\uv_{i'})\check{I}(\vv{q},z,\uv_{i'})\,d\va'
\label{homog}
\end{equation}
for $\vv{q}\in\Rm^2$, $z\in\Rm$, $\uv_i\in\Sm^2$. Let $\rrf{\uvk}$ be the operator which rotates the reference frame in such a way that the $z$-axis is rotated to the direction of $\uvk$ \cite{Machida15}. In ADO for the one-dimensional transport equation, the separated solution in which the spatial variable only exists in the exponential function was used \cite{Barichello11,Barichello-Siewert99a,Barichello-Garcia-Siewert00,Barichello-Siewert02}. In the same spirit, we assume the following separated solution.
\begin{equation}
\check{I}(\vv{q},z,\uv_i)=
\rrf{\uvk(\nu,\vv{q})}\Phi_{\nu}^m(\uv_i)e^{-\hat{k}_z(\nu q)z/\nu},
\label{modes:separatedI}
\end{equation}
where
\be
\Phi_{\nu}^m(\uv_i)=
\phi^m(\nu,\mu_i)\left(1-\mu_i^2\right)^{|m|/2}e^{im\va}.
\ee
The function $\phi^m(\nu,\mu_i)$ is normalized as
\begin{equation}
\sum_{i=1}^{2N}w_i\phi^m(\nu,\mu_i)\left(1-\mu_i^2\right)^{|m|}=1.
\label{normcond}
\end{equation}
The normalization condition (\ref{normcond}) implies
\be
\frac{1}{2\pi}\sum_{i=1}^{2N}w_i\int_0^{2\pi}\phi^m(\nu,\uv_i\cdot\uvk)
\left[1-(\uv_i\cdot\uvk)^2\right]^{|m|}\,d\va=1.
\ee

We note that
\begin{equation}
\rrf{\uvk}Y_{lm}(\uv)=\sum_{m'=-l}^le^{-im'\va_{\uvk}}d_{m'm}^l(\theta_{\uvk})Y_{lm'}(\uv),
\label{rotWig}
\end{equation}
where $d_{m'm}^l$ are Wigner's $d$-matrices \cite{Varshalovich}. If $\uvk\in\Rm^3$, the operator $\rrf{\uvk}$ means the usual rotation of the reference frame in $\Rm^3$ such that the $z$-axis coincides with the direction of the vector $\uvk$ \cite{Markel04}. For $\uvk\in\Cm^3$, the polar angle $\theta_{\uvk}$ and azimuthal angle $\va_{\uvk}$ of $\uvk$ in the laboratory frame are analytically continued as follows \cite{Panasyuk06,Machida-etal10}. We write
\be
\cos\theta_{\uvk}=\hvv{z}\cdot\uvk=\hat{k}_z,\quad
\sin\theta_{\uvk}=\sqrt{1-\cos^2\theta_{\uvk}}=\sqrt{\hat{k}_x^2+\hat{k}_y^2},
\ee
and
\be
\cos\va_{\uvk}=\frac{\hat{k}_x}{\sqrt{\hat{k}_x^2+\hat{k}_y^2}},\quad
\sin\va_{\uvk}=\frac{\hat{k}_y}{\sqrt{\hat{k}_x^2+\hat{k}_y^2}}.
\ee
We take square roots such that $0\le\mathop{\mathrm{arg}}\sqrt{z}<\pi$ and $\Im\sqrt{z}\ge0$ for all $z\in\Cm$ by drawing the branch cut on the positive real axis. The angles $\theta_{\uvk},\va_{\uvk}$ of $\uvk$ are given by
\be
\cos\theta_{\uvk(\nu,\vv{q})}=\hat{k}_z(\nu q),\quad
\sin\theta_{\uvk(\nu,\vv{q})}=i|\nu q|,
\ee
and
\be
\va_{\uvk(\nu,\vv{q})}=\left\{\begin{aligned}
\va_{\vv{q}}+\pi
&\quad(\nu>0),
\\
\va_{\vv{q}}
&\quad(\nu<0).
\end{aligned}\right.
\ee
We can write
\be
d_{mm'}^l\left(\theta_{\uvk(\nu,\vv{q})}\right)=d_{mm'}^l[i\tau(\nu q)].
\ee

Using $Y_{lm}^*(\uv)=(-1)^mY_{l,-m}(\uv)$, we have
\be
\rrf{\uvk}Y_{lm}^*(\uv)=
(-1)^m\sum_{m'=-l}^le^{-im'\va_{\uvk}}d_{m',-m}^l(\theta_{\uvk})Y_{lm'}(\uv)=
\sum_{m'=-l}^le^{im'\va_{\uvk}}d_{m'm}^l(\theta_{\uvk})Y_{lm'}^*(\uv).
\ee
Moreover, we have
\be
\irrf{\uvk}Y_{lm}(\uv)=\sum_{m'=-l}^le^{im'\va_{\uvk}}d_{mm'}^l(\theta_{\uvk})Y_{lm'}(\uv),
\ee
\be
\irrf{\uvk}Y_{lm}^*(\uv)=\sum_{m'=-l}^le^{-im'\va_{\uvk}}d_{mm'}^l(\theta_{\uvk})Y_{lm'}^*(\uv).
\ee
The identity $\irrf{\uvk}\rrf{\uvk}Y_{lm}(\uv)=Y_{lm}(\uv)$ can be derived with the formula $\sum_{m'=-l}^ld_{m'm}^l(\theta_{\uvk})d_{m'm''}^l(\theta_{\uvk})=\delta_{mm''}$.

We note that $\mu_i=\uv_i\cdot\hvv{z}$, where $\hvv{z}={^t}(0,0,1)$. We have
\begin{equation}
\rrf{\uvk(\nu,\vv{q})}\mu_i=\uv_i\cdot\uvk(\nu,\vv{q})=
-i\nu\vv{\omega}_i\cdot\vv{q}+\hat{k}_z(\nu q)\mu_i.
\label{rrfmu}
\end{equation}
This relation is also obtained as follows using $d_{00}^1(\theta)=\cos\theta$, $d_{10}^1(\theta)=-(1/\sqrt{2})\sin\theta$, and $d_{mm'}^l(\theta)=(-1)^{m+m'}d_{-m,-m'}^l(\theta)$:
\ba
\rrf{\uvk(\nu,\vv{q})}\mu_i
&=
\sqrt{\frac{4\pi}{3}}\rrf{\uvk(\nu,\vv{q})}Y_{10}(\uv_i)
\\
&=
\sqrt{\frac{4\pi}{3}}\sum_{m'=-1}^1e^{-im'\va_{\uvk}}d_{m'0}^1(\theta_{\uvk})Y_{1m'}(\uv_i)
\\
&=
\sqrt{1-\mu_i^2}\sin\theta_{\uvk}\cos(\va-\va_{\uvk})+\mu_i\cos\theta_{\uvk}
\\
&=
\hat{k}_z(\nu q)\mu_i-i\nu q\sqrt{1-\mu_i^2}\cos(\va-\va_{\vv{q}}),
\ea

By substitution we have
\begin{equation}
\left(1-\frac{\uv_i\cdot\uvk(\nu,\vv{q})}{\nu}\right)
\rrf{\uvk(\nu,\vv{q})}\Phi_{\nu}^m(\uv_i)=
\varpi\sum_{i'=1}^{2N}w_{i'}\int_0^{2\pi}p(\uv_i,\uv_{i'})
\rrf{\uvk(\nu,\vv{q})}\Phi_{\nu}^m(\uv_{i'})\,d\va'.
\label{modes:phieq}
\end{equation}
Now, we view (\ref{modes:phieq}) in the reference frame whose $z$-axis is 
rotated to the direction of $\uvk$. Note that $p(\uv_i,\uv_{i'})$ is invariant under rotation because it depends only on $\uv_i\cdot\uv_{i'}$. We note by (\ref{rrfmu})
\be
\uv_i\cdot\uvk(\nu,\vv{q})=\rrf{\uvk(\nu,\vv{q})}\uv_i\cdot\hvv{z},
\ee
where $\hvv{z}={^t}(0,0,1)$. If we operate $\irrf{\uvk(\nu,\vv{q})}$ on the left-hand side of (\ref{modes:phieq}), we obtain
\ba
\irrf{\uvk(\nu,\vv{q})}\left(1-\frac{\uv_i\cdot\uvk(\nu,\vv{q})}{\nu}\right)
\rrf{\uvk(\nu,\vv{q})}\Phi_{\nu}^m(\uv_i)
&=
\irrf{\uvk(\nu,\vv{q})}\rrf{\uvk(\nu,\vv{q})}\left(1-\frac{\uv_i\cdot\hvv{z}}{\nu}\right)\Phi_{\nu}^m(\uv_i)
\\
&=
\left(1-\frac{\uv_i\cdot\hvv{z}}{\nu}\right)\Phi_{\nu}^m(\uv_i).
\ea
Let us rewrite the right-hand side of (\ref{modes:phieq}) as
\ba
&
\varpi\sum_{i'=1}^{2N}w_{i'}\int_0^{2\pi}p(\uv_i,\uv_{i'})
\rrf{\uvk(\nu,\vv{q})}\Phi_{\nu}^m(\uv_{i'})\,d\va'
\\
&=
\varpi\sum_{i'=1}^{2N}w_{i'}\int_0^{2\pi}p\left(\rrf{\uvk(\nu,\vv{q})}\uv_i,\rrf{\uvk(\nu,\vv{q})}\uv_{i'}\right)
\rrf{\uvk(\nu,\vv{q})}\Phi_{\nu}^m(\uv_{i'})\,d\va'
\\
&\approx
\varpi\sum_{i'=1}^{2N}w_{i'}\int_0^{2\pi}p\left(\rrf{\uvk(\nu,\vv{q})}\uv_i,\uv_{i'}\right)
\Phi_{\nu}^m(\uv_{i'})\,d\va'.
\ea
In the last step of the above equation, we used the fact that the integral over all angles is invariant under the rotation of the reference frame. Similarly by operating $\irrf{\uvk(\nu,\vv{q})}$ on the right-hand side of (\ref{modes:phieq}), we have
\ba
&
\irrf{\uvk(\nu,\vv{q})}
\varpi\sum_{i'=1}^{2N}w_{i'}\int_0^{2\pi}p(\uv_i,\uv_{i'})
\rrf{\uvk(\nu,\vv{q})}\Phi_{\nu}^m(\uv_{i'})\,d\va'
\\
&\approx
\irrf{\uvk(\nu,\vv{q})}
\varpi\sum_{i'=1}^{2N}w_{i'}\int_0^{2\pi}p\left(\rrf{\uvk(\nu,\vv{q})}\uv_i,\uv_{i'}\right)
\Phi_{\nu}^m(\uv_{i'})\,d\va'
\\
&=
\varpi\sum_{i'=1}^{2N}w_{i'}\int_0^{2\pi}p(\uv_i,\uv_{i'})
\Phi_{\nu}^m(\uv_{i'})\,d\va'.
\ea
Thus by inverse rotation, Eq.~(\ref{modes:phieq}) is reduced to
\begin{equation}
\left(1-\frac{\mu_i}{\nu}\right)\Phi_{\nu}^m(\uv_i)=
\varpi\sum_{l'=0}^{l_{\rm max}}\sum_{m'=-l'}^{l'}{\rm g}^{l'}Y_{l'm'}(\uv_i)
\sum_{i'=1}^{2N}w_{i'}\int_0^{2\pi}Y_{l'm'}^*(\uv_{i'})\Phi_{\nu}^m(\uv_{i'})\,d\va'
\label{homo1d}
\end{equation}
for $i=1,\dots,2N$. We have
\begin{equation}
\sum_{i'=1}^{2N}w_{i'}\int_0^{2\pi}Y_{l'm'}^*(\uv_{i'})\Phi_{\nu}^m(\uv_{i'})\,d\va'=
\delta_{mm'}(-1)^m\sqrt{(2l'+1)\pi}g_{l'}^m(\nu),
\label{intYPhi}
\end{equation}
where
\be
g_l^m(\nu)=
(-1)^m\sqrt{\frac{(l-m)!}{(l+m)!}}\sum_{i=1}^{2N}w_i
\phi^m(\nu,\mu_i)\left(1-\mu_i^2\right)^{|m|/2}P_l^m(\mu_i).
\ee
Noting that $P_l^{-m}(\mu_i)=(-1)^m[(l-m)!/(l+m)!]P_l^m(\mu_i)$ and 
$P_m^m(\mu_i)=(-1)^m(2m-1)!!(1-\mu_i^2)^{m/2}$ ($m\ge0$), we obtain
\be
g_m^m(\nu)=\frac{(2m-1)!!}{\sqrt{(2m)!}}=\frac{\sqrt{(2m)!}}{2^mm!},\quad
g_m^{-m}(\nu)=(-1)^mg_m^m(\nu)
\ee
for $m\ge0$. Equation (\ref{homo1d}) is written as
\be
\left(1-\frac{\mu_i}{\nu}\right)\phi^m(\nu,\mu_i)\left(1-\mu_i^2\right)^{|m|/2}
=
\frac{\varpi}{2}(-1)^m\sum_{l'=|m|}^{l_{\rm max}}(2l'+1){\rm g}^{l'}
\sqrt{\frac{(l'-m)!}{(l'+m)!}}P_{l'}^m(\mu_i)g_{l'}^m(\nu).
\ee

Let us define
\be
p_l^m(\mu)=
(-1)^m\sqrt{\frac{(l-m)!}{(l+m)!}}P_l^m(\mu)\left(1-\mu^2\right)^{-|m|/2}.
\ee
We have
\be
p_m^m(\mu)=\frac{(2m-1)!!}{\sqrt{(2m)!}}=\frac{\sqrt{(2m)!}}{2^mm!},\quad
p_m^{-m}(\mu)=(-1)^mp_m^m(\mu)
\ee
for $m\ge0$. Moreover from $(l-m+1)P_{l+1}^m(\mu)=(2l+1)\mu P_l^m(\mu)-(l+m)P_{l-1}^m(\mu)$,
\be
\sqrt{(l+1)^2-m^2}p_{l+1}^m(\mu)=(2l+1)\mu p_l^m(\mu)-\sqrt{l^2-m^2}p_{l-1}^m(\mu)
\ee
for $l\ge|m|+1$, and
\be
p_{|m|+1}^m(\mu)=\sqrt{2|m|+1}\mu p_{|m|}^m(\mu).
\ee
Thus,
\begin{equation}
\left(\nu-\mu_i\right)\phi^m(\nu,\mu_i)=
\frac{\varpi\nu}{2}\sum_{l'=|m|}^{l_{\rm max}}(2l'+1){\rm g}^{l'}
p_{l'}^m(\mu_i)g_{l'}^m(\nu).
\label{eigenmode1}
\end{equation}
We note that
\be
g_l^m(\nu)=\sum_{i=1}^{2N}w_i\phi^m(\nu,\mu_i)p_l^m(\mu_i)\left(1-\mu_i^2\right)^{|m|}.
\ee
The following recurrence relations are obtained.
\ba
&
\sqrt{(l+1)^2-m^2}g_{l+1}^m(\nu)+\sqrt{l^2-m^2}g_{l-1}^m(\nu)
\\
&=
(2l+1)\nu g_l^m(\nu)-
\frac{\varpi\nu}{2}\sum_{l'=|m|}^{l_{\rm max}}(2l+1)(2l'+1){\rm g}^{l'}g_{l'}^m(\nu)
\sum_{i=1}^{2N}w_ip_l^m(\mu_i)p_{l'}^m(\mu_i)\left(1-\mu_i^2\right)^{|m|}.
\ea
Recall $\int_{-1}^1P_l^m(\mu)P_{l'}^m(\mu)\,d\mu=2(l+m)!\delta_{ll'}/[(2l+1)(l-m)!]$. For sufficiently large $N$, the recurrence relations below are obtained.
\be
\sqrt{(l+1)^2-m^2}g_{l+1}^m(\nu)+\sqrt{l^2-m^2}g_{l-1}^m(\nu)=
\nu h_lg_l^m(\nu)
\ee
for $l\ge|m|+1$, and
\be
\sqrt{2|m|+1}g_{|m|+1}^m(\nu)=\nu h_{|m|}g_{|m|}^m(\nu),
\ee
where
\be
h_l=\left\{\begin{aligned}
(2l+1)\left(1-\varpi{\rm g}^l\right),
&\quad 0\le l\le l_{\rm max},
\\
2l+1,&\quad l>l_{\rm max}.
\end{aligned}\right..
\ee
The polynomials $g_l^m(\nu)$ are called the normalized Chandrasekhar polynomials \cite{Garcia-Siewert89,Garcia-Siewert90}.

For numerical calculation, $g_l^m(\nu)$ can be computed using the recurrence relations. To compute $g_l^m(\nu)$ for $\nu>1$, we define \cite{Garcia-Siewert90}
\be
\bar{g}_l^m(\nu)=\frac{g_{l+1}^m(\nu)}{g_l^m(\nu)}.
\ee
Then we have
\be
\bar{g}_{l-1}^m(\nu)=
\frac{\sqrt{l^2-m^2}}{\nu h_l-\sqrt{(l+1)^2-m^2}\bar{g}_l^m(\nu)}
\ee
for $l=|m|+1,|m|+2,\dots$. We can set $\bar{g}_l^m(\nu)=0$ for sufficiently large $l$. Using $\bar{g}_l^m(\nu)$, we obtain
\be
g_{l+1}^m(\nu)=\bar{g}_l^m(\nu)g_l^m(\nu),
\quad l=|m|,|m|+1,\dots.
\ee

Suppose $\nu\neq\mu_i$. Then from (\ref{eigenmode1}),
\be
\phi^m(\nu,\mu_i)=
\frac{\varpi\nu}{2}\frac{g^m(\nu,\mu_i)}{\nu-\mu_i},
\ee
where
\be
g^m(\nu,\mu_i)=
\sum_{l'=|m|}^{l_{\rm max}}(2l'+1){\rm g}^{l'}p_{l'}^m(\mu_i)g_{l'}^m(\nu).
\ee
We obtain
\begin{equation}
\Phi_{\nu}^m(\uv_i)=
\frac{(-1)^m\varpi\nu}{\nu-\mu_i}\sum_{l=|m|}^{l_{\rm max}}\sqrt{(2l+1)\pi}{\rm g}^lg_l^m(\nu)Y_{lm}(\uv_i).
\label{modes:modes0}
\end{equation}
If $\nu\neq\uv_i\cdot\uvk$, we have
\be
\phi^m\left(\nu,\uv_i\cdot\uvk(\nu,\vv{q})\right)=
\frac{\varpi\nu}{2}\frac{g^m\left(\nu,\uv_i\cdot\uvk(\nu,\vv{q})\right)}{\nu-\uv_i\cdot\uvk(\nu,\vv{q})}.
\ee
Thus,
\begin{equation}
\rrf{\uvk(\nu,\vv{q})}\Phi_{\nu}^m(\uv_i)=
\frac{(-1)^m\varpi\nu}{\nu-\uv_i\cdot\uvk(\nu,\vv{q})}\sum_{l=|m|}^{l_{\rm max}}\sqrt{(2l+1)\pi}{\rm g}^lg_l^m(\nu)
\left(\rrf{\uvk(\nu,\vv{q})}Y_{lm}(\uv_i)\right).
\label{modes:modes}
\end{equation}

The above eigenmodes satisfy the orthogonality relation.

\begin{lem}
\label{modes:lem:orth}
\begin{equation}
\sum_{i=1}^{2N}w_i\mu_i\int_0^{2\pi}
\left(\rrf{\uvk(\nu,\vv{q})}\Phi_{\nu}^m(\uv_i)\right)
\left(\rrf{\uvk(\nu',\vv{q})}\Phi_{\nu'}^{m'*}(\uv_i)\right)\,d\va=
\mathcal{N}(\nu,q)\delta_{\nu\nu'}\delta_{mm'},
\label{lemeq1}
\end{equation}
where $\mathcal{N}(\nu,q)$ is a positive constant which depends on $\nu,q$.
\end{lem}

The factor $\mathcal{N}(\nu,q)$ in Lemma \ref{modes:lem:orth} satisfies $\mathcal{N}(-\nu,q)=-\mathcal{N}(\nu,q)$ because $g^m(-\nu,-\mu_i)=g^m(\nu,\mu_i)$. Note that
\be
\left|\Phi_{\nu}^m(\uv_i)\right|^2=
\frac{\pi\varpi^2\nu^2}{(\nu-\mu_i)^2}\left[\sum_{l=|m|}^{l_{\rm max}}
\sqrt{2l+1}{\rm g}^lY_{lm}(\mu_i,0)g_l^m(\nu)\right]^2,
\ee
and for $\uvk=\uvk(\nu,\vv{q})$,
\ba
\irrf{\uvk}\mu_i
&=
\sqrt{\frac{4\pi}{3}}\irrf{\uvk}Y_{10}(\uv_i)=
\sqrt{\frac{4\pi}{3}}\sum_{m'=-1}^1d_{0m'}^1(\vth_{\uvk})Y_{1m'}(\uv_i)
\\
&=
\hat{k}_z(\nu q)\mu_i-i|\nu q|\sqrt{1-\mu_i^2}\cos\va.
\ea
Since integrals over all angles do not change their values if the reference frame is rotated, for sufficiently large $N$ we have
\ba
&
\sum_{i=1}^{2N}w_i\mu_i\int_0^{2\pi}
\left(\rrf{\uvk(\nu,\vv{q})}\Phi_{\nu}^m(\uv_i)\right)
\left(\rrf{\uvk(\nu,\vv{q})}\Phi_{\nu}^{m*}(\uv_i)\right)\,d\va
\\
&=
\sum_{i=1}^{2N}w_i\int_0^{2\pi}\mu_i\rrf{\uvk(\nu,\vv{q})}
\left(\Phi_{\nu}^m(\uv_i)\Phi_{\nu}^{m*}(\uv_i)\right)\,d\va
\\
&\approx
\sum_{i=1}^{2N}w_i\int_0^{2\pi}\irrf{\uvk(\nu,\vv{q})}\left[\mu_i\rrf{\uvk(\nu,\vv{q})}
\left(\Phi_{\nu}^m(\uv_i)\Phi_{\nu}^{m*}(\uv_i)\right)\right]\,d\va
\\
&=
\sum_{i=1}^{2N}w_i\int_0^{2\pi}
\left(\irrf{\uvk(\nu,\vv{q})}\mu_i\right)
\left|\Phi_{\nu}^m(\uv_i)\right|^2\,d\va
\ea
Hence we have
\ba
\mathcal{N}(\nu,q)
&\approx
\sum_{i=1}^{2N}w_i\int_0^{2\pi}
\left(\irrf{\uvk(\nu,\vv{q})}\mu_i\right)
\left|\Phi_{\nu}^m(\uv_i)\right|^2\,d\va
\\
&=
2\pi\hat{k}_z(\nu q)\sum_{i=1}^{2N}w_i\mu_i
\left|\Phi_{\nu}^m(\uv_i)\right|^2.
\ea

\begin{proof}[Proof of Lemma \ref{modes:lem:orth}]
We begin by expressing (\ref{modes:phieq}) as
\ba
&
\left(1-\frac{\uv_i\cdot\uvk(\nu,\vv{q})}{\nu}\right)
\rrf{\uvk(\nu,\vv{q})}\Phi_{\nu}^m(\uv_i)
\\
&=
\varpi\sum_{l=0}^{l_{\rm max}}\sum_{m'=-l}^lg^lY_{lm'}(\uv_i)
\sum_{i'=1}^{2N}w_{i'}\int_0^{2\pi}Y_{lm'}^*(\uv_{i'})
\rrf{\uvk(\nu,\vv{q})}\Phi_{\nu}^m(\uv_{i'})\,d\va'.
\ea
Note that $\nu=\nu^m$ depends on $m$. Let us consider $\uvk_1=\uvk(\nu_1^{m_1},\vv{q})$ and $\uvk_2=\uvk(\nu_2^{m_2},\vv{q})$. We have
\ba
&
\left(\rrf{\uvk_2}\Phi_{\nu_2}^{m_2}(\uv_i)\right)
\rrf{\uvk_1}\left(1-\frac{\mu_i}{\nu_1}\right)\Phi_{\nu_1}^{m_1}(\uv_i)
\\
&=
\left(\rrf{\uvk_2}\Phi_{\nu_2}^{m_2}(\uv_i)\right)
\varpi\sum_{l=0}^{l_{\rm max}}\sum_{m=-l}^lg^lY_{lm}(\uv_i)
\sum_{i'=1}^{2N}w_{i'}\int_0^{2\pi}Y_{lm}^*(\uv_{i'})
\rrf{\uvk_1}\Phi_{\nu_1}^{m_1}(\uv_{i'})\,d\va',
\ea
and
\ba
&
\left(\rrf{\uvk_1}\Phi_{\nu_1}^{m_1}(\uv_i)\right)
\rrf{\uvk_2}\left(1-\frac{\mu_i}{\nu_2}\right)\Phi_{\nu_2}^{m_2}(\uv_i)
\\
&=
\left(\rrf{\uvk_1}\Phi_{\nu_1}^{m_1}(\uv_i)\right)
\varpi\sum_{l=0}^{l_{\rm max}}\sum_{m=-l}^lg^lY_{lm}^*(\uv_i)
\sum_{i'=1}^{2N}w_{i'}\int_0^{2\pi}Y_{lm}(\uv_{i'})
\rrf{\uvk_2}\Phi_{\nu_2}^{m_2}(\uv_{i'})\,d\va'.
\ea
By subtraction,
\ba
&
-\frac{1}{\nu_1}\left(\rrf{\uvk_2}\Phi_{\nu_2}^{m_2}(\uv_i)\right)
\rrf{\uvk_1}\mu_i\Phi_{\nu_1}^{m_1}(\uv_i)+
\frac{1}{\nu_2}\left(\rrf{\uvk_1}\Phi_{\nu_1}^{m_1}(\uv_i)\right)
\rrf{\uvk_2}\mu_i\Phi_{\nu_2}^{m_2}(\uv_i)
\\
&=
\left(\rrf{\uvk_2}\Phi_{\nu_2}^{m_2}(\uv_i)\right)
\varpi\sum_{l=0}^{l_{\rm max}}\sum_{m=-l}^lg^lY_{lm}(\uv_i)
\sum_{i'=1}^{2N}w_{i'}\int_0^{2\pi}Y_{lm}^*(\uv_{i'})
\rrf{\uvk_1}\Phi_{\nu_1}^{m_1}(\uv_{i'})\,d\va'
\\
&-
\left(\rrf{\uvk_1}\Phi_{\nu_1}^{m_1}(\uv_i)\right)
\varpi\sum_{l=0}^{l_{\rm max}}\sum_{m=-l}^lg^lY_{lm}^*(\uv_i)
\sum_{i'=1}^{2N}w_{i'}\int_0^{2\pi}Y_{lm}(\uv_{i'})
\rrf{\uvk_2}\Phi_{\nu_2}^{m_2}(\uv_{i'})\,d\va'.
\ea

Hence by summing both sides of the above two equations by $\sum_{i=1}^{2N}w_i\int_0^{2\pi}\,d\va$,
\ba
\left(\frac{\hat{k}_z(\nu_2q)}{\nu_2}-\frac{\hat{k}_z(\nu_1q)}{\nu_1}\right)
\sum_{i=1}^{2N}w_i\int_0^{2\pi}\mu_i\left(\rrf{\uvk_1}\Phi_{\nu_1}^{m_1}(\uv_i)\right)\left(\rrf{\uvk_2}\Phi_{\nu_2}^{m_2}(\uv_i)\right)\,d\va=0.
\ea
Suppose $\nu=\nu_1=\nu_2$ but $m_1\neq m_2$. In this case, $\uvk=\uvk_1=\uvk_2$ and
\ba
&
\sum_{i=1}^{2N}w_i\int_0^{2\pi}\mu_i\left(\rrf{\uvk}\Phi_{\nu}^{m_1}(\uv_i)\right)\left(\rrf{\uvk}\Phi_{\nu}^{m_2}(\uv_i)\right)\,d\va
\\
&=
\sum_{i=1}^{2N}w_i\int_0^{2\pi}\mu_i\left(\rrf{\uvk}\Phi_{\nu}^{m_1}(\uv_i)\Phi_{\nu}^{m_2}(\uv_i)\right)\,d\va
\propto\delta_{m_1,-m_2}.
\ea
Noting that $\nu^{-m}=\nu^m$ and $\Phi_{\nu}^{-m}(\uv_i)=\Phi_{\nu}^{m*}(\uv_i)$, we have
\be
\sum_{i=1}^{2N}w_i\int_0^{2\pi}\mu_i\left(\rrf{\uvk_1}\Phi_{\nu_1}^{m_1}(\uv_i)\right)\left(\rrf{\uvk_2}\Phi_{\nu_2}^{m_2*}(\uv_i)\right)\,d\va
\propto
\delta_{\nu_1\nu_2}\delta_{m_1m_2}.
\ee
\end{proof}

Let us calculate eigenvalues. Equation (\ref{homo1d}) can be rewritten as ($i=1,\dots,2N$, $-l_{\rm max}\le m\le l_{\rm max}$)
\be
\left(1-\frac{\mu_i}{\nu}\right)\phi^m(\nu,\mu_i)=
\frac{\varpi}{2}\sum_{l'=|m|}^{l_{\rm max}}{\rm g}^{l'}(2l'+1)p_{l'}^m(\mu_i)
\sum_{i'=1}^{2N}w_{i'}p_{l'}^m(\mu_{i'})\left(1-\mu_{i'}^2\right)^{|m|}\phi^m(\nu,\mu_{i'}).
\ee
Hence for $i=1,\dots,N$,
\ba
\left(1-\frac{\mu_i}{\nu}\right)\phi^m(\nu,\mu_i)
&=
\frac{\varpi}{2}\sum_{l'=|m|}^{l_{\rm max}}{\rm g}^{l'}(2l'+1)p_{l'}^m(\mu_i)
\\
&\times
\sum_{i'=1}^Nw_{i'}p_{l'}^m(\mu_{i'})\left(1-\mu_{i'}^2\right)^{|m|}
\left[\phi^m(\nu,\mu_{i'})+(-1)^{l'+m}\phi^m(\nu,-\mu_{i'})\right],
\ea
\ba
\left(1+\frac{\mu_i}{\nu}\right)\phi^m(\nu,-\mu_i)
&=
\frac{\varpi}{2}\sum_{l'=|m|}^{l_{\rm max}}{\rm g}^{l'}(2l'+1)p_{l'}^m(\mu_i)
\\
&\times
\sum_{i'=1}^Nw_{i'}p_{l'}^m(\mu_{i'})\left(1-\mu_{i'}^2\right)^{|m|}
\left[(-1)^{l'+m}\phi^m(\nu,\mu_{i'})+\phi^m(\nu,-\mu_{i'})\right].
\ea
We arrive at the following matrix-vector equation.
\ba
\left(I_N-\frac{1}{\nu}\Xi\right)\vv{\Phi}_+^m(\nu)
&=
\frac{\varpi}{2}\left[W_+^m\vv{\Phi}_+^m(\nu)+W_-^m\vv{\Phi}_-^m(\nu)\right],
\\
\left(I_N+\frac{1}{\nu}\Xi\right)\vv{\Phi}_-^m(\nu)
&=
\frac{\varpi}{2}\left[W_-^m\vv{\Phi}_+^m(\nu)+W_+^m\vv{\Phi}_-^m(\nu)\right],
\ea
where $I_N$ is the $N$-dimensional identity matrix. Matrix $\Xi$ and vectors $\vv{\Phi}_{\pm}^m$ are defined as
\be
\Xi
=\begin{pmatrix}\mu_1&& \\ &\ddots& \\ &&\mu_N\end{pmatrix},
\quad
\vv{\Phi}_{\pm}^m(\nu)
=\begin{pmatrix}
\phi^m(\nu,\pm\mu_1) \\ \vdots\\ \phi^m(\nu,\pm\mu_N)
\end{pmatrix}.
\ee
Elements of the matrices $W_{\pm}^m$ are given by
\ba
\{W_{\pm}^m\}_{ij}
&=
w_j\sum_{l=|m|}^{l_{\rm max}}(2l+1){\rm g}^lp_l^m(\pm\mu_i)p_l^m(\mu_j)
\left(1-\mu_j^2\right)^{|m|}
\\
&=
w_j\sum_{l=|m|}^{l_{\rm max}}(2l+1){\rm g}^l\frac{(l-m)!}{(l+m)!}
P_l^m(\pm\mu_i)P_l^m(\mu_j)
\left(1-\mu_i^2\right)^{-|m|/2}\left(1-\mu_j^2\right)^{|m|/2}
\\
&=
4\pi w_j\sum_{l=|m|}^{l_{\rm max}}{\rm g}^l
\frac{\left(1-\mu_j^2\right)^{|m|/2}}{\left(1-\mu_i^2\right)^{|m|/2}}
Y_{lm}(\pm\mu_i,0)Y_{lm}(\mu_j,0),
\ea
where we wrote $Y_{lm}(\uv_i)=Y_{lm}(\mu_i,\va)$. The fact $\vv{\Phi}_{\pm}^m(-\nu)=\vv{\Phi}_{\mp}^m(\nu)$ implies that $-\nu$ is an eigenvalue if $\nu$ is an eigenvalue.

According to \cite{Barichello11,Siewert00}, we introduce
\be
\vv{U}^m(\nu)=\vv{\Phi}_+^m(\nu)+\vv{\Phi}_-^m(\nu),\quad
\vv{V}^m(\nu)=\vv{\Phi}_+^m(\nu)-\vv{\Phi}_-^m(\nu).
\ee
By adding and subtracting two equations we obtain
\ba
&
\vv{U}^m(\nu)-\frac{1}{\nu}\Xi\vv{V}^m(\nu)=\frac{\varpi}{2}(W_+^m+W_-^m)\vv{U}^m(\nu),
\\
&
\vv{V}^m(\nu)-\frac{1}{\nu}\Xi\vv{U}^m(\nu)=\frac{\varpi}{2}(W_+^m-W_-^m)\vv{V}^m(\nu).
\ea
Hence we arrive at \cite{Barichello-Siewert00,Siewert-Wright99}
\begin{equation}
E_-^mE_+^m\Xi\vv{U}^m=\frac{1}{\nu^2}\Xi\vv{U}^m,
\label{evalsys}
\end{equation}
where
\be
E_{\pm}^m=\left[I_N-\frac{\varpi}{2}(W_+^m\pm W_-^m)\right]\Xi^{-1}.
\ee
We note that (\ref{evalsys}) is not the only matrix-vector equation for $1/\nu^2$ (see \cite{Barichello-Siewert99b}). From (\ref{evalsys}), the eigenvalues are obtained as
\be
\nu=\nu_n^m>0\quad(n=1,\dots,N,\;m=0,\dots,l_{\rm max})
\ee
with the relation $\nu_n^{-m}=\nu_n^m$. In addition, $-\nu_n^m$ are also eigenvalues.

\section{Scattering term}
\label{Ispart}

Let $\va_0$ be the azimuthal angle of $\uv_{i_0}$ for some $i_0$. By setting $\widetilde{S}=\delta(z-z')\delta_{ii_0}\delta(\va-\va_0)$ in (\ref{intro:I2do}) for $z'\in\Rm$, let us consider the fundamental solution $G_{\vv{q}}(z,\uv;z',\uv')$:
\begin{equation}
\begin{aligned}
&
\left(\mu_i\frac{\pp}{\pp z}+1+i\vv{\omega}_i\cdot\vv{q}\right)G_{\vv{q}}(z,\uv_i;z',\uv_{i_0})
\\
&=
\varpi\sum_{i'=1}^{2N}w_{i'}\int_0^{2\pi}p(\uv_i,\uv_{i'})G_{\vv{q}}(z,\uv_{i'};z',\uv_{i_0})\,d\va+
\delta(z-z')\delta_{ii_0}\delta(\va-\va_0)
\label{GreenRTE}
\end{aligned}
\end{equation}
for $z\in\Rm$, $1\le i\le 2N$, $\va\in[0,2\pi)$. We obtain (see Appendix \ref{green})
\begin{equation}
\begin{aligned}
G_{\vv{q}}(z,\uv_i;z',\uv_{i_0})
&=
\sum_{m=-l_{\rm max}}^{l_{\rm max}}\sum_{n=1}^N
\frac{w_{i_0}}{\mathcal{N}(\nu_n^m,q)}e^{\mp\hat{k}_z(\nu_n^mq)(z-z')/\nu_n^m}
\\
&\times
\left(\rrf{\uvk_{\pm}}\Phi_{\pm\nu_n^m}^m(\uv_i)\right)
\left(\rrf{\uvk_{\pm}}\Phi_{\pm\nu_n^m}^{m*}(\uv_{i_0})\right)
\end{aligned}
\label{gfunc1}
\end{equation}
where upper (lower) signs are taken for $z>z'$ ($z<z'$) and
\be
\uvk_{\pm}=\uvk(\pm\nu_n^m,\vv{q}).
\ee
Let us define
\be
C(\tau;\zeta,\eta)=\frac{e^{-\tau/\zeta}-e^{-\tau/\eta}}{\zeta-\eta}.
\ee
We note that $C(-\tau;-\zeta,-\eta)=-C(\tau;\zeta,\eta)$ and
\be
C(\tau;\eta,\eta)=\frac{\tau}{\eta^2}.
\ee
The solution $\check{I}_s$ to (\ref{intro:I2do}) is obtained as
\begin{equation}
\begin{aligned}
&
\check{I}_s(\vv{q},z,\uv_i)=
\sum_{i'=1}^{2N}\int_0^{2\pi}\int_{-\infty}^{\infty}G_{\vv{q}}(z,\uv_i;z',\uv_{i'})\widetilde{S}(\vv{q},z',\uv_{i'})\,dz'd\va'
\\
&=
\frac{\varpi\mu_t^2}{|\mu_0|}
\int_{-\infty}^{\infty}\Theta(\mu_0z')\left[\sum_{i'=1}^{2N}
\int_0^{2\pi}G_{\vv{q}}(z,\uv_i;z',\uv_{i'})p(\uv_{i'},\uv_{i_0})\,d\va'\right]
e^{-(1+i\vv{q}\cdot\vv{\omega}_0)z'/\mu_0}\,dz'.
\end{aligned}
\label{eqIs2}
\end{equation}
Since $p(\uv_{i'},\uv_{i_0})=p(\rrf{\uvk}\uv_{i'},\rrf{\uvk}\uv_{i_0})$ for a unit vector $\uvk\in\Cm$, we have
\be
\sum_{l=0}^{l_{\rm max}}\sum_{m=-l}^l{\rm g}^lY_{lm}(\uv_{i'})Y_{lm}^*(\uv_{i_0})=
\sum_{l=0}^{l_{\rm max}}\sum_{m=-l}^l{\rm g}^l\left(\rrf{\uvk}Y_{lm}(\uv_{i'})\right)\left(\rrf{\uvk}Y_{lm}^*(\uv_{i_0})\right).
\ee
We obtain
\ba
&
\sum_{i'=1}^{2N}\int_0^{2\pi}G_{\vv{q}}(z,\uv_i;z',\uv_{i'})
p(\uv_{i'},\uv_{i_0})\,d\va'
\\
&=
\sum_{m=-l_{\rm max}}^{l_{\rm max}}\sum_{n=1}^N\frac{1}{\mathcal{N}(\nu_n^m,q)}
e^{\mp\hat{k}_z(\nu_n^mq)(z-z')/\nu_n^m}
\left(\rrf{\uvk_{\pm}}\Phi_{\pm\nu_n^m}^m(\uv_i)\right)
\\
&\times
\sum_{l=0}^{l_{\rm max}}\sum_{m'=-l}^l{\rm g}^l\left(\rrf{\uvk_{\pm}}Y_{lm'}^*(\uv_{i_0})\right)
\sum_{i'=1}^{2N}w_{i'}\int_0^{2\pi}
\left(\rrf{\uvk_{\pm}}\Phi_{\pm\nu_n^m}^{m*}(\uv_{i'})\right)
\left(\rrf{\uvk_{\pm}}Y_{lm'}(\uv_{i'})\right)\,d\va'
\\
&\approx
\sum_{m=-l_{\rm max}}^{l_{\rm max}}\sum_{n=1}^N\frac{1}{\mathcal{N}(\nu_n^m,q)}
e^{\mp\hat{k}_z(\nu_n^mq)(z-z')/\nu_n^m}
\left(\rrf{\uvk_{\pm}}\Phi_{\pm\nu_n^m}^m(\uv_i)\right)
\\
&\times
\sum_{l=0}^{l_{\rm max}}\sum_{m'=-l}^l{\rm g}^l\left(\rrf{\uvk_{\pm}}Y_{lm'}^*(\uv_{i_0})\right)
\sum_{i'=1}^{2N}w_{i'}\int_0^{2\pi}\Phi_{\pm\nu_n^m}^{m*}(\uv_{i'})Y_{lm'}(\uv_{i'})\,d\va'.
\ea
To reach the rightmost side of the above equation, we used the fact that the discretized integral over all angles is numerically unchanged about the rotation of the reference frame for sufficiently large $N$. Hence,
\begin{equation}
\begin{aligned}
&
\sum_{i'=1}^{2N}\int_0^{2\pi}G_{\vv{q}}(z,\uv_i;z',\uv_{i'})
p(\uv_{i'},\uv_{i_0})\,d\va'
\\
&\approx
\sum_{m=-l_{\rm max}}^{l_{\rm max}}\sum_{n=1}^N\frac{1}{\mathcal{N}(\nu_n^m,q)}
e^{\mp\hat{k}_z(\nu_n^mq)(z-z')/\nu_n^m}
\left(\rrf{\uvk_{\pm}}\Phi_{\pm\nu_n^m}^m(\uv_i)\right)
\\
&\times
\sum_{l=0}^{l_{\rm max}}\sum_{m'=-l}^l{\rm g}^l\rrf{\uvk_{\pm}}Y_{lm'}^*(\uv_{i_0})
\sum_{i'=1}^{2N}w_{i'}\int_0^{2\pi}\Phi_{\pm\nu_n^m}^{m*}(\uv_{i'})Y_{lm'}(\uv_{i'})\,d\va'
\\
&=
\sum_{m=-l_{\rm max}}^{l_{\rm max}}\sum_{n=1}^N\frac{(-1)^m}{\mathcal{N}(\nu_n^m,q)}
e^{\mp\hat{k}_z(\nu_n^mq)(z-z')/\nu_n^m}
\left(\rrf{\uvk_{\pm}}\Phi_{\pm\nu_n^m}^m(\uv_i)\right)
\\
&\times
\sum_{l=|m|}^{l_{\rm max}}\sqrt{(2l+1)\pi}{\rm g}^lg_l^m(\pm\nu_n^m)
\left(\rrf{\uvk_{\pm}}Y_{lm}^*(\uv_{i_0})\right),
\end{aligned}
\label{intGp}
\end{equation}
where we used (\ref{intYPhi}). 

Hereafter we assume
\begin{equation}
\uv_0=\hvv{z}={^t}(0,0,1).
\label{uv0isz}
\end{equation}
Since
\be
\rrf{\uvk_{\pm}}Y_{lm}(\hvv{z})=
\sqrt{\frac{2l+1}{4\pi}}d_{0m}^l[i\tau(\nu_nq)],
\ee
we have
\be
\rrf{\uvk_{\pm}}Y_{lm}^*(\uv_{i_0})=
\rrf{\uvk(\pm\nu_n^m,\vv{q})}Y_{lm}^*(\hvv{z})=
\sqrt{\frac{2l+1}{4\pi}}d_{0m}^l[i\tau(\nu_n^mq)].
\ee
If (\ref{intGp}) is substituted into (\ref{eqIs2}), $\check{I}_s(\vv{q},z,\uv_i)$ is calculated as
\ba
&
\check{I}_s(\vv{q},z,\uv_i)
\\
&=
\frac{\varpi\mu_t^2}{2}
\int_{-\infty}^z\Theta(z')\Biggl[
\sum_{m=-l_{\rm max}}^{l_{\rm max}}\sum_{n=1}^N\frac{(-1)^m}{\mathcal{N}(\nu_n^m,q)}
e^{-\hat{k}_z(\nu_n^mq)(z-z')/\nu_n^m}
\left(\rrf{\uvk_+}\Phi_{\nu_n^m}^m(\uv_i)\right)
\\
&\times
\sum_{l=|m|}^{l_{\rm max}}(2l+1){\rm g}^lg_l^m(\nu_n^m)d_{0m}^l[i\tau(\nu_n^mq)]
\Biggr]
e^{-z'/\mu_0}\,dz'
\\
&+
\frac{\varpi\mu_t^2}{2}
\int_z^{\infty}\Theta(z')\Biggl[
\sum_{m=-l_{\rm max}}^{l_{\rm max}}\sum_{n=1}^N\frac{(-1)^m}{\mathcal{N}(\nu_n^m,q)}
e^{\hat{k}_z(\nu_n^mq)(z-z')/\nu_n^m}
\left(\rrf{\uvk_-}\Phi_{-\nu_n^m}^m(\uv_i)\right)
\\
&\times
\sum_{l=|m|}^{l_{\rm max}}(2l+1){\rm g}^lg_l^m(-\nu_n^m)d_{0m}^l[i\tau(\nu_n^mq)]
\Biggr]
e^{-z'/\mu_0}\,dz'.
\ea
In the above equation, we assumed that $N$ is large enough that the approximation in (\ref{intGp}) is numerically exact. Thus we have
\begin{equation}
\check{I}_s(\vv{q},z,\uv_i)=
\frac{\varpi\mu_t^2}{2}\sum_{m=-l_{\rm max}}^{l_{\rm max}}\sum_{n=1}^N\frac{(-1)^m\nu_n^m}{\mathcal{N}(\nu_n^m,q)\hat{k}_z(\nu_n^mq)}
\left[\Theta(z)\left(\psi_1+\psi_2\right)+\left(1-\Theta(z)\right)\psi_3\right],
\label{search:psip}
\end{equation}
where
\be
\psi_1=
C\left(z;1,\frac{\nu_n^m}{\hat{k}_z(\nu_n^mq)}\right)
\left(\rrf{\uvk_+}\Phi_{\nu_n^m}^m(\uv_i)\right)
\sum_{l=|m|}^{l_{\rm max}}(2l+1){\rm g}^lg_l^m(\nu_n^m)d_{0m}^l[i\tau(\nu_nq)],
\ee
\be
\psi_2=
\frac{e^{-z}\hat{k}_z(\nu_n^mq)}{\nu_n^m+\hat{k}_z(\nu_n^mq)}
\left(\rrf{\uvk_-}\Phi_{-\nu_n^m}^m(\uv_i)\right)
\sum_{l=|m|}^{l_{\rm max}}(2l+1){\rm g}^lg_l^m(-\nu_n^m)
d_{0m}^l[i\tau(\nu_nq)],
\ee
\be
\psi_3=
\frac{e^{\hat{k}_z(\nu_n^mq)z/\nu_n^m}\hat{k}_z(\nu_n^mq)}{\nu_n^m+\hat{k}_z(\nu_n^mq)}
\left(\rrf{\uvk_-}\Phi_{-\nu_n^m}^m(\uv_i)\right)
\sum_{l=|m|}^{l_{\rm max}}(2l+1){\rm g}^lg_l^m(-\nu_n^m)
d_{0m}^l[i\tau(\nu_nq)].
\ee

\section{Energy density}
\label{eden}

In this section we return to the original unit of length and put the subscript $*$ for the scaled space variables. We suppose $\vv{\rho}\neq\vv{0}$. The energy density $U(\vv{r})$ is given by
\begin{equation}
U(\vv{r})=\int_{\Sm^2}I(\vv{r},\uv)\,d\uv.
\label{defU}
\end{equation}
Noting $\widetilde{I}_s\approx\check{I}_s$ by the discretization of the cosine $\mu$ of the polar angle, we have
\ba
U(\vv{r})
&=
\int_{\Sm^2}I_b(\vv{r},\uv)\,d\uv+
\frac{1}{(2\pi)^2}\int_{\Rm^2}e^{i\vv{q}\cdot\vv{\rho}\mu_t}\int_{\Sm^2}\widetilde{I}_s(\vv{q},z_*,\uv)\,d\uv d\vv{q}
\\
&\approx
\frac{1}{(2\pi)^2}\int_{\Rm^2}e^{i\vv{q}\cdot\vv{\rho}\mu_t}\sum_{i=1}^{2N}w_i\int_0^{2\pi}\check{I}_s(\vv{q},z_*,\uv_i)\,d\va d\vv{q}
\\
&=
\frac{1}{2\pi}\int_0^{\infty}qJ_0(q\rho_*)\sum_{i=1}^{2N}w_i\int_0^{2\pi}\check{I}_s(\vv{q},z_*,\uv_i)\,d\va dq
\\
&=
U(\rho,z),
\ea
where $\rho_*=\mu_t\rho$, $z_*=\mu_tz$, and we used the Hansen-Bessel formula
\be
J_0(x)=\frac{1}{2\pi}\int_0^{2\pi}e^{ix\cos\va}\,d\va
\ee
with $J_0(x)$ ($x\ge0$) the Bessel function of order $0$. Note that
\be
\sum_{i=1}^{2N}w_i\int_0^{2\pi}\rrf{\uvk(\nu,\vv{q})}\Phi_{\nu}^m(\uv_i)\,d\va
\approx2\pi\delta_{m0}
\ee
for sufficiently large $N$ because $\int_{\Sm^2}\rrf{\uvk}\Phi_{\nu}^m(\uv)\,d\uv=\int_{\Sm^2}\Phi_{\nu}^m(\uv)\,d\uv$. Within this approximation we have
\begin{equation}
U(\rho,z)=
\frac{\varpi\mu_t^2}{2}\int_0^{\infty}qJ_0(q\rho_*)F(q,z_*)\,dq,
\label{Uintq}
\end{equation}
where
\ba
F(q,z_*)
&=
\sum_{n=1}^N\frac{\nu_n^0}{\mathcal{N}(\nu_n^0,q)}
\Biggl[
\Theta(z_*)\frac{e^{-z_*}-e^{-\hat{k}_z(\nu_n^0q)z_*/\nu_n^0}}{\hat{k}_z(\nu_n^0q)-\nu_n^0}
\sum_{l=0}^{l_{\rm max}}(2l+1){\rm g}^lg_l^0(\nu_n^0)d_{00}^l[i\tau(\nu_nq)]
\\
&+
\Theta(z_*)\frac{e^{-z_*}}{\hat{k}_z(\nu_n^0q)+\nu_n^0}
\sum_{l=0}^{l_{\rm max}}(-1)^l(2l+1){\rm g}^lg_l^0(\nu_n^0)
d_{00}^l[i\tau(\nu_nq)]
\\
&+
\left(1-\Theta(z_*)\right)
\frac{e^{\hat{k}_z(\nu_n^0q)z_*/\nu_n^0}}{\hat{k}_z(\nu_n^0q)+\nu_n^0}
\sum_{l=0}^{l_{\rm max}}(-1)^l(2l+1){\rm g}^lg_l^0(\nu_n^0)
d_{00}^l[i\tau(\nu_nq)]
\Biggr].
\ea

For the purpose of numerical calculation, we use the asymptotic expression of the Bessel function:
\ba
&
J_0(x)=\sqrt{\frac{2}{\pi x}}
\\
&\times
\left[
\left(1-\frac{9}{128x^2}+O(x^{-4})\right)\cos\left(x-\frac{\pi}{4}\right)+
\left(\frac{1}{8x}-\frac{75}{1024x^3}+O(x^{-5})\right)\sin\left(x-\frac{\pi}{4}\right)
\right].
\ea
We define
\ba
&
d(q,\rho_*)=
qJ_0(q\rho_*)-\sqrt{\frac{2q}{\pi\rho_*}}
\\
&\times
\left[
\left(1-\frac{9}{128(q\rho_*)^2}\right)\cos\left(q\rho_*-\frac{\pi}{4}\right)+
\left(\frac{1}{8q\rho_*}-\frac{75}{1024(q\rho_*)^3}\right)\sin\left(q\rho_*-\frac{\pi}{4}\right)
\right].
\ea
Let us set
\be
a=\frac{\pi}{4\rho_*}.
\ee
We have
\begin{equation}
\begin{aligned}
U(\rho,z)
&=
\frac{\varpi\mu_t^2}{2}\int_0^aqJ_0(q\rho_*)F(q,z_*)\,dq+
\frac{\varpi\mu_t^2}{2}\int_a^{\infty}d(q,\rho_*)F(q,z_*)\,dq
\\
&+
\frac{\varpi\mu_t^2}{\sqrt{2\pi\rho_*}}\int_0^{\infty}\left[F_1(s,\vv{r}_*)+F_2(s,\vv{r}_*)\right]\,ds,
\end{aligned}
\label{final_U}
\end{equation}
where $s=q-a$,
\ba
F_1(s,\vv{r}_*)
&=
\sqrt{s+\frac{\pi}{4\rho_*}}\left(
1-\frac{9}{128(s\rho_*+\frac{\pi}{4})^2}\right)
F\left(s+\frac{\pi}{4\rho_*},z_*\right)\cos(s\rho_*),
\\
F_2(s,\vv{r}_*)
&=
\sqrt{s+\frac{\pi}{4\rho_*}}\left(
\frac{1}{8(s\rho_*+\frac{\pi}{4})}-\frac{75}{1024(s\rho_*+\frac{\pi}{4})^3}
\right)
F\left(s+\frac{\pi}{4\rho_*},z_*\right)\sin(s\rho_*).
\ea
The integrals $\int_0^{\infty}F_1(q)\,dq$ and $\int_0^{\infty}F_2(q)\,dq$ can be evaluated by the double-exponential formula \cite{Ooura-Mori91,Ooura-Mori99,Ogata05}. Define
\be
\phi(t)=\frac{t}{1-e^{-6\sinh{t}}}
\ee
with
\be
\phi'(t)=\frac{1-(1+6t\cosh{t})e^{-6\sinh{t}}}{\left(1-e^{-6\sinh{t}}\right)^2}.
\ee
With the approximation of the double-exponential formula, we can write
\ba
&
\int_0^{\infty}F_1(s,\vv{r}_*)\,ds\approx
\frac{\pi}{\rho_*}\sum_{k=-N_k}^{N_k}F_1\left(\frac{\pi}{h\rho_*}\phi\left(kh+\frac{h}{2}\right),\vv{r}_*\right)
\phi'\left(kh+\frac{h}{2}\right),
\\
&
\int_0^{\infty}F_2(s,\vv{r}_*)\,ds\approx
\frac{\pi}{\rho_*}\sum_{k=-N_k}^{N_k}F_2\left(\frac{\pi}{h\rho_*}\phi(kh),\vv{r}_*\right)
\phi'(kh),
\ea
where $N_k>0$ is an integer and $h$ is a mesh size.

\section{Numerical tests}
\label{num}

Let us compute $U(\vv{r})=U(\rho,z)=\int_{\Sm^2}I(\vv{r},\uv)\,d\uv$ in (\ref{defU}), where $I(\vv{r},\uv)$ is the solution of the radiative transport equation (\ref{rte0}) for the source term $f(\vv{r},\uv)=\delta(\vv{\rho})\delta(z)\delta(\uv-\uv_0)$ given in (\ref{source1}). As was assumed in (\ref{uv0isz}), $\uv_0=\hvv{z}$.

\begin{figure}[ht]
\begin{center}
\includegraphics[width=0.45\textwidth]{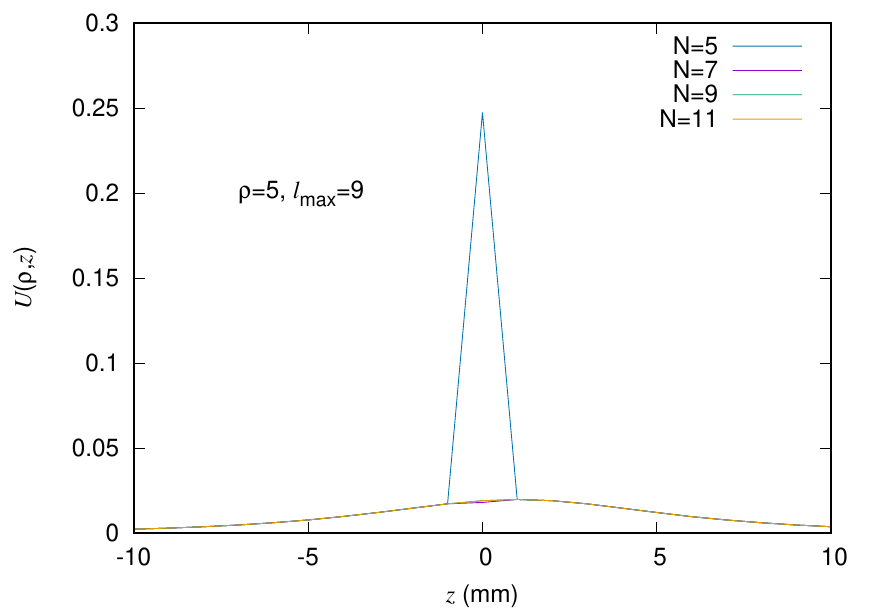}
\includegraphics[width=0.45\textwidth]{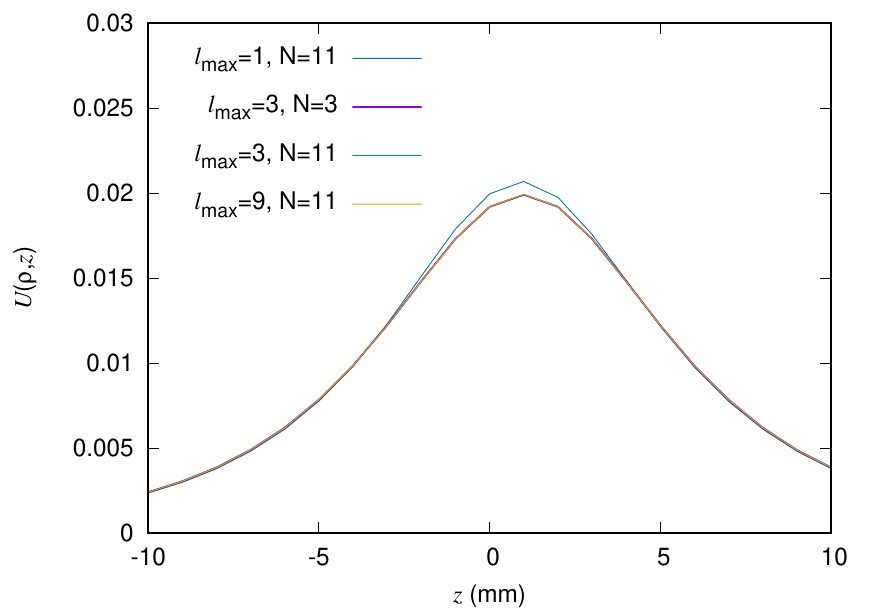}
\end{center}
\caption{
The energy density $U(\rho,z)$ in (\ref{final_U}) is plotted for $\rho=5\,{\rm mm}$. The parameters are set to $\mu_a=0.01$, $\mu_s=10$, and ${\rm g}=0.9$. (Left) Dependence on $N$ for $l_{\rm max}=9$. (Right) Convergence behavior.
}
\label{figure3}
\end{figure}

\begin{figure}[ht]
\begin{center}
\includegraphics[width=0.45\textwidth]{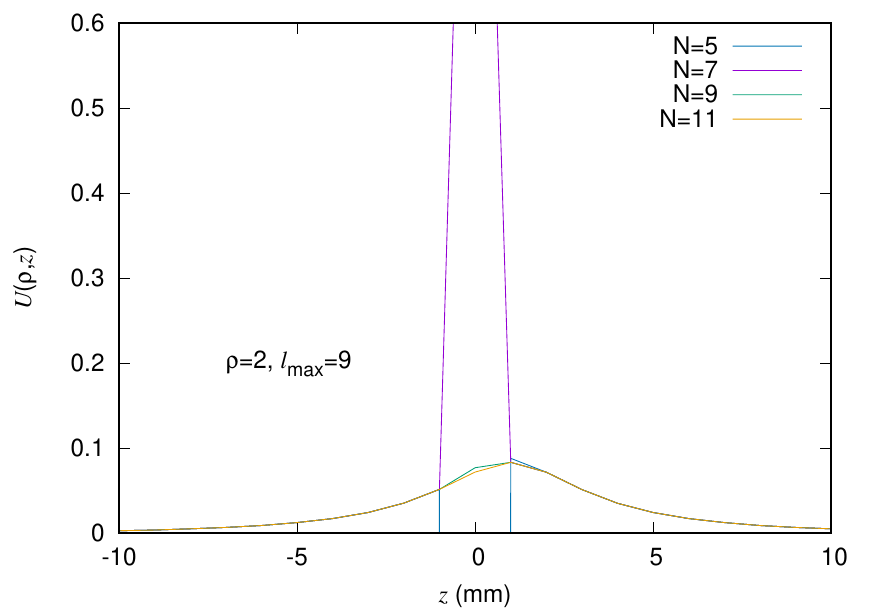}
\includegraphics[width=0.45\textwidth]{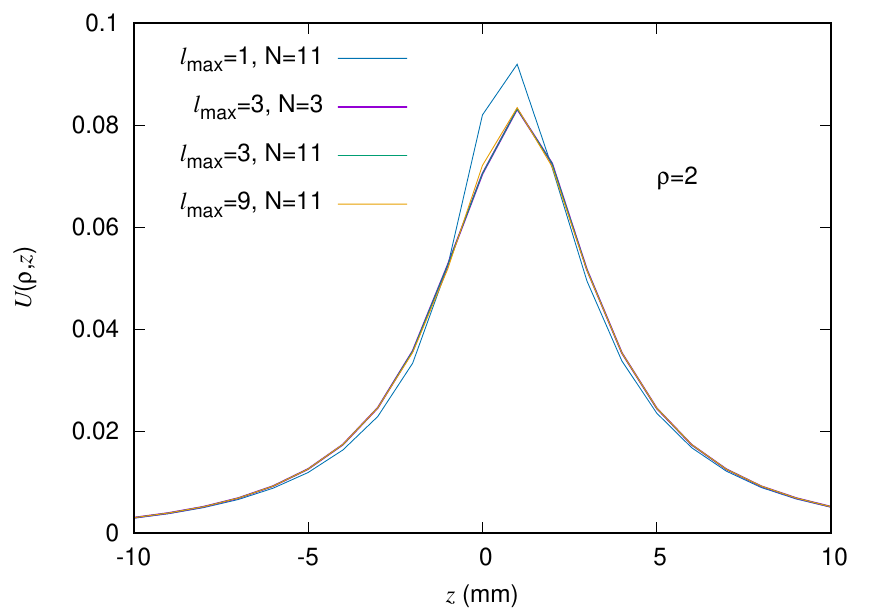}
\end{center}
\caption{
Same as Fig.~\ref{figure3} but $\rho=2\,{\rm mm}$. (Left) Dependence on $N$ for $l_{\rm max}=9$. (Right) Convergence behavior.
}
\label{figure4}
\end{figure}

Numerical calculation was conducted for
\be
\mu_a=0.01\,{\rm mm}^{-1},\quad\mu_s=10\,{\rm mm}^{-1},\quad
{\rm g}=0\;\mbox{or}\;0.9.
\ee
When the integral (\ref{Uintq}) over $q$ is computed for $U(\rho,z)$ making use of the double-exponential formula, we set $h=0.01$, and $N_k=400$. The trapezoidal rule was used for the first two integrals in (\ref{final_U}) with $10^4$ trapezoids. The upper limit of the second integral was set to $1$. We took $101$ points on the $z$-axis between $-50\,{\rm mm}$ and $50\,{\rm mm}$ (${\rm g}=0.9$) or $-5\,{\rm mm}$ and $5\,{\rm mm}$ (${\rm g}=0$). Results are shown in Figs.~\ref{figure3} through \ref{figure2}. 

We set ${\rm g}=0.9$ in Figs.~\ref{figure3} and \ref{figure4}. We first set $\rho=5\,{\rm mm}$. In Fig.~\ref{figure3}(a), $U(\rho,z)$ in (\ref{final_U}) is plotted for $l_{\rm max}=9$. A smooth curve is obtained when $N=9$ and $N=11$. Figure \ref{figure3}(b) shows that the calculation with $l_{\rm max}=3$ and $N=3$ is almost identical to the result for $l_{\rm max}=9$ and $N=11$. Using FORTRAN (gfortran), the computation time was $5.6\,{\rm sec}$ for $l_{\rm max}=9$, $N=11$ and $0.5\,{\rm sec}$ for $l_{\rm max}=3$, $N=3$ on a laptop computer (Apple, MacBook Pro, $2.3\,{\rm GHz}$ Intel Core i5). Next we set $\rho=2\,{\rm mm}$. Figure \ref{figure4}(a) shows that only the case of $N\ge11$ is successful when $l_{\rm max}=9$. In Fig.~\ref{figure4}(b), the result for $l_{\rm max}=3$, $N=3$ is close to the curve for $l_{\rm max}=9$, $N=11$, which gives the converged curve.

In Figs.~\ref{figure1} and \ref{figure2}, $U(\rho,z)$ in (\ref{final_U}) is compared to Monte Carlo simulation. The Monte-Carlo RTE solver MC \cite{MarkelMC} was used with the number of photons $10^8$. The Monte Carlo simulation uses the Henyey-Greenstein model ($l_{\rm max}=\infty$). Parameters were chosen as $l_{\rm max}=3$, $N=3$. In Fig.~\ref{figure1}, we set ${\rm g}=0.9$. In Fig.~\ref{figure2}, ${\rm g}$ is set to $0$. In both figures, $U(\rho,z)$ are plotted as functions of $z$ for different $\rho$. Good agreement is seen in each line of Figs.~\ref{figure1} and \ref{figure2}.

\begin{figure}[ht]
\begin{center}
\includegraphics[width=0.7\textwidth]{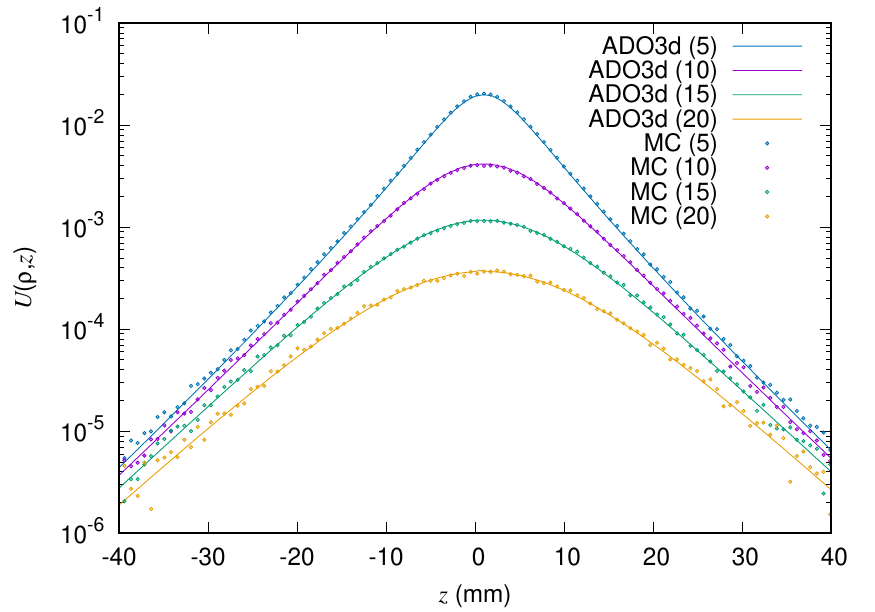}
\end{center}
\caption{
The energy density $U(\rho,z)$ is plotted. Results from (\ref{final_U}) are compared to Monte Carlo simulation. From the top, $\rho=5$, $10$, $15$, and $20\,{\rm mm}$. The parameters are set to $\mu_a=0.01$, $\mu_s=10$, and ${\rm g}=0.9$.
}
\label{figure1}
\end{figure}

\begin{figure}[ht]
\begin{center}
\includegraphics[width=0.7\textwidth]{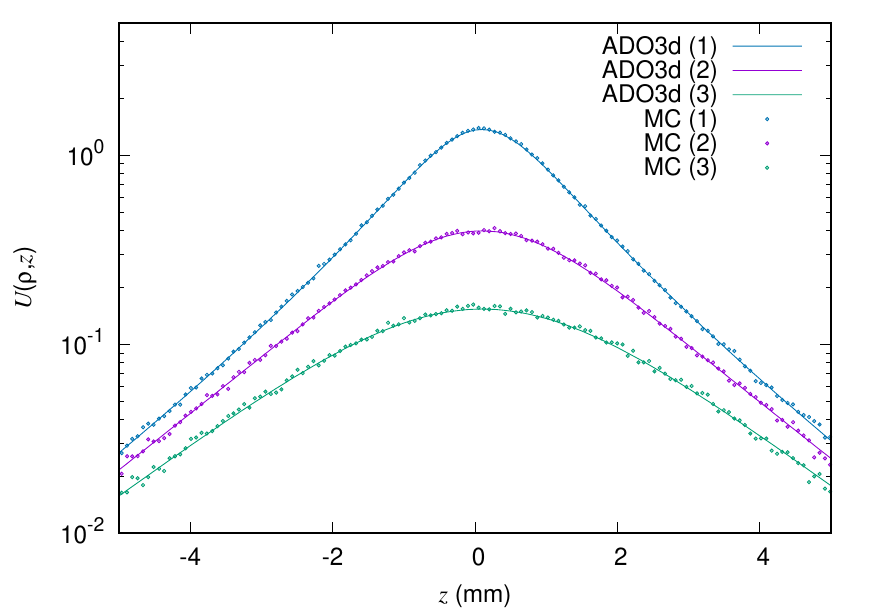}
\end{center}
\caption{
Same as Fig.~\ref{figure1}, but ${\rm g}=0$. From the top, $\rho=1$, $2$, and $3\,{\rm mm}$.
}
\label{figure2}
\end{figure}

\section{Concluding remarks}
\label{concl}

In one dimension, the equivalence between the discrete-ordinates method and spherical-harmonics method is known \cite{Case-Zweifel,Barichello-Siewert98}. In three dimensions, eigenmodes of the method of rotated reference frames \cite{Markel04} are obtained by expanding singular eigenfunctions with spherical harmonics \cite{Barichello-Garcia-Siewert98} and then rotating the resulting expressions. However, the series expansion in the method of rotated reference frames diverges because components of $\uvk$ become very large even if $\uvk\cdot\uvk=1$ always holds.

Kim and Keller \cite{Kim-Keller03} and Kim \cite{Kim04} established discrete-ordinates with plane-wave decomposition for the three-dimensional radiative transport equation.  Their method can be viewed as separation of variables assuming the form
\be
\check{I}(\vv{q},z,\uv_i)=V_{\lambda}(\uv_i,\vv{q})e^{\lambda(\vv{q})z}.
\ee
The eigenvalues $\lambda$ and eigenvectors $V_{\lambda}$ are numerically computed for each $\vv{q}\in\Rm^2$. In the present paper, on the other hand, the form (\ref{modes:separatedI}) was assumed. The explicit form of the eigenvectors is obtained and we can compute eigenvalues $\nu$ by diagonalizing a matrix once.

In Sec.~\ref{num}, we computed the energy density. For the computation of the specific intensity itself, the post-processing procedure \cite{Chandrasekhar50,Siewert00}, which was proposed in the original ADO is expected to help obtaining stable numerical solutions.

In this paper, the pencil beam (\ref{source1}) was assumed. For a general source function $f(\vv{r},\uv)$, the solution $I(\vv{r},\uv)$ to (\ref{rte0}) is given by
\be
I(\vv{r},\uv)=\int_{\Sm^2}\int_{\Rm^3}G(\vv{r},\uv;\vv{r}',\uv')f(\vv{r}',\uv')\,d\vv{r}'d\uv',
\ee
where $G(\vv{r},\uv;\vv{r}',\uv')$ ($\vv{r}'={^t}(\vv{\rho}',z')$) is the solution of (\ref{rte0}) for the source term $\delta(\vv{\rho}-\vv{\rho}')\delta(z-z')\delta(\uv-\uv')$. For any $\uv_0$, we can obtain $G(\vv{r},\uv;\vv{r}',\uv_0)$ simply by shifting $\vv{\rho},z$ by $\vv{\rho}',z'$ in the solution $I(\vv{r},\uv)$ to (\ref{rte0}) for the source term (\ref{source1}). Moreover, the solution to (\ref{rte0}) can be computed without the decomposition (\ref{Idecomp}) if the source term is smooth. In this alternative approach for smooth $f$, the solution to (\ref{rte0}) is obtained as $I_s(\vv{r},\uv)$ in (\ref{subt:rte}) in which $S(\vv{r},\uv)$ is replaced by $f(\vv{r},\uv)$. Accordingly, $\widetilde{S}$ in (\ref{eqIs2}) is replaced by the Fourier transform of $f$. Then the calculation in Sec.~\ref{Ispart} is modified in a straightforward way.

%\section*{Acknowledgements}

\section*{Disclosure statement}

No potential conflict of interest was reported by the authors.

\section*{Funding}

This work was supported by JSPS KAKENHI Grant Number JP17K05572 and JST PRESTO Grant Number JPMJPR2027.

\appendix

\section{Fundamental solution}
\label{green}

Let us obtain $G_{\vv{q}}(z,\uv_i;z',\uv_{i_0})$ which satisfies (\ref{GreenRTE}). Since $G_{\vv{q}}\to0$ as $z\to\pm\infty$, we can write
\ba
&
G_{\vv{q}}(z,\uv_i;z',\uv_{i_0})
\\
&=
\left\{\begin{aligned}
&
\sum_{m=-l_{\rm max}}^{l_{\rm max}}\sum_{n=1}^NA_n^m
\rrf{\uvk(\nu_n^m,\vv{q})}\Phi_{\nu_n^m}^m(\uv_i)e^{-\hat{k}_z(\nu_n^mq)(z-z')/\nu_n^m},
\\
&\qquad z>z',
\\
&
\sum_{m=-l_{\rm max}}^{l_{\rm max}}\sum_{n=1}^NB_n^m
\rrf{\uvk(-\nu_n^m,\vv{q})}\Phi_{-\nu_n^m}^m(\uv_i)e^{\hat{k}_z(\nu_n^mq_0)(z-z')/\nu_n^m},
\\
&\qquad z<z',
\end{aligned}\right.
\ea
where coefficients $A_n^m=A_n^m(\vv{q},\uv_{i_0})$, $B_n^m=B_n^m(\vv{q},\uv_{i_0})$ will be determined from the source term. With the source term, the jump condition is written as
\ba
&
\mu_i\left[G_{\vv{q}}(z'+0,\uv_i;z',\uv_{i_0})-G_{\vv{q}}(z'-0,\uv_i;z',\uv_{i_0})\right]
\\
&=
\delta_{ii_0}\delta(\va-\va_0).
\ea
Hence,
\ba
&
\mu_i\sum_{m=-l_{\rm max}}^{l_{\rm max}}\sum_{n=1}^N
\\
&
\left[A_n^m\rrf{\uvk(\nu_n^m,\vv{q})}\Phi_{\nu_n^m}^m(\uv_i)-B_n^m\rrf{\uvk(-\nu_n^m,\vv{q})}\Phi_{-\nu_n^m}^m(\uv_i)\right]
\\
&=
\delta_{ii_0}\delta(\va-\va_0).
\ea
Using the orthogonality relation in Lemma (\ref{modes:lem:orth}), we obtain
\ba
A_n^m&=
\frac{w_{i_0}}{\mathcal{N}(\nu_n^m,q)}\rrf{\uvk(\nu_n^m,\vv{q})}\Phi_{\nu_n^m}^{m*}(\uv_{i_0}),
\\
B_n^m&=
\frac{-w_{i_0}}{\mathcal{N}(-\nu_n^m,q)}\rrf{\uvk(-\nu_n^m,\vv{q})}\Phi_{-\nu_n^m}^{m*}(\uv_{i_0}).
\ea
Thus we arrive at (\ref{gfunc1}).

\end{document}